\documentclass[11pt]{article}

\usepackage{amsmath,amsfonts,amssymb,amsthm,amscd}
\usepackage{extarrows}
\usepackage{ulem}
\usepackage{color}
\title{On groups with definable $f$-generics definable in $p$-adically closed fields}
\date{\today}

\author{Anand Pillay \\
  \text{University of Notre Dame} \and Ningyuan Yao \\
  \text{Fudan University}}

\newtheorem{Theorem}{Theorem}
\newtheorem{Thm}{Theorem}[section]
\newtheorem{Prop}[Thm]{Proposition}
\newtheorem{Def}[Thm]{Definition}
\newtheorem{Rmk}[Thm]{Remark}
\newtheorem{Lemma}[Thm]{Lemma}
\newtheorem{Cor}[Thm]{Corollary}
\newtheorem{Fact}[Thm]{Fact}

\newtheorem{Question}{Question}
\newtheorem{Conjecture}{Conjecture}

\newtheorem*{Claim}{Claim}

\newcommand{\R}{\mathbb R}
\newcommand{\Q}{{\mathbb Q}_p}
\newcommand{\Z}{{\mathbb Z}_p}

\newcommand{\N}{\mathbb N}

\newcommand{\C}{\mathbb C}

\newcommand{\M}{\mathbb M}

\newcommand{\K}{\mathbb K}
\newcommand{\Ga}{{\mathbb G}_a}
\newcommand{\Gm}{{\mathbb G}_m}
\newcommand{\V}{{\cal O}}

\newcommand{\ra}{\longrightarrow}
\newcommand{\lra}{\longrightarrow}
\newcommand{\sq}{\subseteq}

\DeclareMathOperator{\Th}{Th}
\DeclareMathOperator{\im}{im}
\DeclareMathOperator{\id}{id}

\DeclareMathOperator{\tp}{tp}

\DeclareMathOperator{\dcl}{dcl}
\DeclareMathOperator{\st}{st}
\DeclareMathOperator{\Ker}{Ker}
\DeclareMathOperator{\cl}{cl}
\DeclareMathOperator{\acl}{acl}

\DeclareMathOperator{\rad}{rad}

\def\blue{\textcolor{blue}}

\begin{document}
\maketitle
\begin{abstract}
The aim of this paper is to develop the theory of groups definable in the $p$-adic field $\Q$, with ``definable $f$-generics" in the sense of an ambient saturated elementary extension of $\Q$. We call such groups \emph{definable $f$-generic} groups.

So, by a ``definable f-generic'' or {\em dfg} group we mean a definable group in a saturated model with a global f-generic type which is definable over a small model.  In the present context the group is definable over $\Q$, and the small model will be $\Q$ itself. The notion of a dfg group is dual, or rather opposite to that of an fsg group (group with ``finitely satisfiable generics") and is a useful tool to describe the analogue of torsion free o-minimal groups in the $p$-adic context.

In the current paper our group will be definable over $\Q$ in an ambient saturated elementary extension $\K$ of $\Q$, so as to make sense of the notions of $f$-generic etc.
In this paper we will show that every definable $f$-generic group definable in $\Q$ is 
virtually isomorphic to a finite index subgroup of a trigonalizable algebraic group over $\Q$. This is analogous to the $o$-minimal context, where every connected torsion free group definable in $\R$ is isomorphic to a trigonalizable algebraic group (Lemma 3.4, \cite{COS}).  We will also show that every open definable $f$-generic subgroup of a definable $f$-generic group has finite index, and every $f$-generic type of a definable $f$-generic group is almost periodic, which gives a positive answer to the problem raised in \cite{P-Y} of whether $f$-generic types coincide with almost periodic types in the $p$-adic case.
\end{abstract}

\section{Introduction}
In the recent years there has been growing interest in the interaction between topological
dynamics and model theory. This approach was introduced by Newelski \cite{Newelski},  then developed in a number of papers, including \cite{Pillay-TD}, \cite{G. Jagiella}, \cite{CS}, \cite{Y-L} and \cite{P-Y}, and now called definable topological dynamics. Definable topological dynamics studies the action of a group $G$ definable in a structure $M$ on its type space $S_G(M)$ and tries to link the invariants suggested by  topological dynamics (e.g.  enveloping semigroups,  minimal subflows, Ellis groups, ...) with model-theoretic invariants.

This framework allows us to generalize stable group theory to some unstable groups with tame behaviour. Under the assumption that the theory $T$ has $NIP$,  the class of definably amenable groups  is a reasonable choice for the class of stable-like groups in $NIP$ context since Newelski's conjecture \cite{Newelski} holds.
Namely, the smallest type-definable subgroup of $G$ of bounded-index, written $G^{00}$, exists and $G/G^{00}$ is isomorphic to its Ellis group (see \cite{CS}). In the stable case, this coincides with the space (group) of   generic types of $G$.

We will now be assuming that our ambient theory $T$ is $NIP$.

    As generic types may not always exist in the $NIP$ environment, other
    notions were introduced to describe the generic-like formulas and types. The notion of $f$-generic was first introduced in \cite{$NIP$II} and then slightly modified in \cite{CS}.
    We let $\M$ be a so-called ``monster" model (of whatever complete theory we are working with). Recall that this means that $\M$ is $\kappa$-saturated and $\kappa$-strongly homomogeneous for some
    suitably large $\kappa$. If the reader is willing to make some set-theoretic assumptions, they can just take $\M$ to be $\kappa$-saturated and of cardinality $\kappa$ for suitable $\kappa$.
    Let $G$ be a group definable in $\M$. A definable subset (or formula) $X\subseteq G$ (definable over $M$) is said to be $f$-generic if for every $g\in G$, the left translate $gX$ does not fork over $M$. A global type $p\in S_G(\M)$ is $f$-generic if every formula in $p$ is $f$-generic. A global $f$-generic type $p$ is strongly  $f$-generic if  every $G$-translate of $p$ does not fork over a fixed small submodel $M$ (equivalently by $NIP$ is invariant under automorphisms fixing $M$).


One of the nice observations in \cite{CS} was that an $NIP$ group is definably amenable if and only if it admits a global  strongly   $f$-generic type. Now there are two extreme case for a global strongly  $f$-generic type $p$ in $NIP$ environment:
\begin{enumerate}
\item [(i)]  $p$ (as well as every left translate) is finitely satisfiable in some fixed small model $M$. In this case $p$ is actually generic, meaning that for every formula $\phi$ in $p$, finitely many translates of $\phi$ cover $G$. We call $p$ a finitely satisfiable generic, abbreviated as \underline{$fsg$}.
\item [(ii)] $p$ (so also all its translates) is definable over some fixed small model $M$, and we call $p$ a definable $f$-generic, abbreviated as \underline{$dfg$}.
\end{enumerate}
A definable group $G$ is said to be an $fsg$ group if it admits a global $fsg$ type, and a $dfg$ group if it admits a global $dfg$ type.

In \cite{$NIP$}, Hrushovski, Peterzil, and Pillay proved that $fsg$ groups coincide with definably compact groups in  $o$-minimal expansions of real closed field. Moreover, in \cite{O-P},  Onshuus and Pillay proved that it is also true for groups definable in $\Q$. The above results say that the model-theoretic concept of $fsg$ is able to describe the topological concept of ``compactness" in $o$-minimal and $p$-adic context.

On the other hand, putting various results together, we know that for definably connected groups definable in $o$-minimal expansions of $RCF$, torsion-free is equivalent to being $dfg$. Here are some details. The left implies right direction is precisely Proposition 4.7 in \cite{CP-connected-component}.
For the right implies left direction: The proof of Lemma 1.15 in \cite{P-Y} shows that if $G$ is $dfg$, then there is a $G^{0}$-invariant definable type, so by connectedness there is a $G$-invariant definable type $p$. Now we know $G$ to be definably amenable, so by Propositions 2.6 and 4.6 of \cite{CP-connected-component}, there is a normal definable torsion-free (solvable) subgroup $H$ of $G$ such that $G/H$ is definably compact (and definably connected). We claim that $G/H$ is trivial, which will show that $G = H$ is torsion-free. Assuming $G/H$ nontrivial, $p$ gives rise to a $G/H$-invariant type $p/H$ of $G/H$, which implies that $(G/H)^{00} = G/H$, a contradiction to definable compactness of $G/H$ (for example to the truth of Pillay's conjecture).

Since torsion-free means ``totally non-definably compact" for groups definable in $o$-minimal expansions of $RCF$ this translates into $dfg$ meaning ``totally non $fsg$" in the $o$-minimal context. This observation is also witnessed by the following:
\begin{Fact}\cite{Simon-dis-nodis}
$\Th(\Q)$ and $o$-minimal theories are distal.
\end{Fact}
Note that distality of a theory implies that no global types can be both definable and finitely satisfiable (in a small model), and in fact such types are weakly orthogonal.

Anyway it is natural to ask what happens in the $p$-adic context. We will be working with $Th(\Q)$ in the ring language and we will actually be restricting ourselves to groups definable in or over the standard model $\Q$.

\begin{Question}\label{Main Question 1}
 Are $dfg$ groups  exactly ``totally noncompact" groups in this $p$-adic context?
\end{Question}
The aim of the paper is to answer this question. Our first result is \blue{a} structure theorem for $dfg$ groups over $\Q$,  which gives a positive answer to Question \ref{Main Question 1}.
\begin{Theorem}\label{Main theorem 1}
Let $G$ be a group definable over $\Q$ with $dfg$, then there is a normal sequence
\[
G_0\vartriangleleft ... \vartriangleleft G_i \vartriangleleft G_{i+1} \vartriangleleft... \vartriangleleft G_n
\]
where the $G_{i}$ are definable over $\Q$, and
such that $G_0$ is finite,  $G_n$ is a finite index subgroup of $G$, and  each $G_{i+1}/G_i$ is definably (over $\Q$) isomorphic to either the additive group $\Ga$, or a finite index subgroup of the multiplicative group $\Gm$.
\end{Theorem}

Recall that if $G$ is a definable group in a structure $M$, then by definition a type in $S_{G}(M)$ is almost periodic if the closure of its $G$-orbit is a minimal $G$-invariant closed subset of $S_{G}(M)$.  In the following, the ambient model $M$ is assumed to be saturated.
The following conjecture was raised in \cite{P-Y}
\begin{Conjecture}\label{Main Conjecture 2}
 Let $G$ be a $dfg$ group definable in an $NIP$ structure. Then any global f-generic  type
is almost periodic.
\end{Conjecture}

The following theorem  gives a positive answer to Conjecture \ref{Main Conjecture 2} in the $p$-adic case. The proof makes use of a recent result in \cite{Y-Trigonalizable}.

\begin{Theorem}
Suppose that $G$ is a $dfg$ group definable over $\Q$, then every global $f$-generic type of $G$ is almost periodic.
\end{Theorem}

The paper is organized as follows. In the rest of this introduction we will recall precise notation, definitions and results  from earlier papers, relevant to our results. In Section 2.1, we will prove some general results for  $dfg$ and $fsg$ groups. In Section 2.2, we will characterise the one-dimensional $dfg$ groups definable over $\Q$.  Section 2.3 contains the key lemmas of the paper, where we proved that every algebraic group containing an open $dfg$ subgroup is trigonalizable. In section 2.4,  we prove a structure theorem for an arbitrary  $dfg$ group definable over $\Q$, showing that every $dfg$ group is virtually trigonalizable over $\Q$, and concluding that every  global $f$-generic type of a $dfg$ group is almost periodic.

\subsection{Notations}
We will assume a basic knowledge of model theory. Good references are \cite{Pzt-book} and \cite{M-book}. Let $T$ be a complete theory in a language $L$ and let $\M$ be the monster model, in which every type over a small subset  $A\subset \M$ is realized, where ``small" means $|A|<|\M|$.
We will assume $T$ to be $1$-sorted. By $x,y,z$ we mean finite tuples of variables,  and $a,b,c\in \M$  finite tuples from $\M$. For a subset $A$ of $\M$, $L_{A}$ is the language obtained from $L$ by adjoining constants for elements of $A$. As a rule, by a formula we will mean   an $L_\M$-formula (unless otherwise stated). For an arbitrary elementary submodel $M$ of $\M$ and an $L_M$-formula $\phi(x)$, $\phi(M)$ denotes the definable subset of $M^{|x|}$ defined by $\phi$, and a set $X\subseteq M^n$ is  definable  if  there is an $L_M$-formula $\phi(x)$  such that $X=\phi(M)$. If $\bar X\subseteq \M^n$ is definable, defined with parameters from $M$, then $\bar X(M)$ will denote $\bar X\cap M^n$, the realizations from $M$, which is clearly a definable subset of $M^n$. Suppose that $X\subseteq \M^n$ is a definable set, defined with parameters from $M$, then we write $S_X(M)$ for the space of complete types concentrating on $X$. We use freely basic notions of model theory such as definable type, heir, coheir, .... The book \cite{Pzt-book} is a possible source.

We will be distinguishing in the current paper, between definable and interpretable. So by a definable set in the saturated structure $\M$ we mean a subset of $\M^{n}$ defined by a formula with parameters.

\subsection{Topological dynamics}

Our reference for (abstract) topological dynamics is \cite{Auslander}.

Given a (Hausdorff) topological group $G$, by a $G$-flow we mean a continuous action $G\times X\to X$ of $G$ on a compact (Hausdorff) topological space $X$. We sometimes write the flow as $(X,G)$. Often it is assumed that there is a dense orbit,
and a $G$-flow $(X,G)$ with a distinguished point $x\in X$ whose orbit is dense is called a 
$G$-ambit (although we will not make much use of this notation).

In spite of $p$-adic algebraic groups being nondiscrete topological groups, we will be treating them as discrete groups so as to have their actions on type spaces being continuous. So in this background section we may assume $G$ to be a discrete group, in which case  a $G$-flow is simply an action of $G$ by homeomorphisms on a compact space $X$.

By a subflow of $(X,G)$ we mean a nonempty closed $G$-invariant nonempty subspace $Y$ of $X$ (together with the action of $G$ on $Y$). $(X,G)$ will always have  minimal nonempty subflows. A point $x\in X$ is almost periodic if the closure of its orbit is a minimal subflow.

Let $(X,G)$ and $(Y,G)$ be flows (with the same acting group). A
homomorphism from $X$ to $Y$ is a continuous map $f: X\longrightarrow Y$  such that $f(gx)=gf(x)$ for all $g\in G$ and $x\in X$.

\begin{Fact}\label{homomorphism of flows}
Let $(X,G)$ and $(Y,G)$ be flows, $f: X\longrightarrow Y$ a homomorphism. Then
\begin{enumerate}
 \item [(i)] If $X_0$ is a minimal subflow of $X$, then $Y_0=f(X_0)$ is a minimal subflow of Y.
 \item  [(ii)] If $Y$ is a minimal flow, then $f$ is onto.
\end{enumerate}
\end{Fact}

\subsection{Definably amenable groups}

Let $T$ be any first-order theory,  $\M\models T$ a very saturated model, and $G$  a group $\emptyset$-definable in $\M$ defined by the formula $G(x)$. For any $M\prec \M$, $G(M)=\{g\in M^n| g\in G\}$ is a subgroup of $G$. By $S_G(M)$  we mean the space of all complete types over $M$ concentrating on $G(x)$.
It is easy to see that $S_G(M)$ together with the type of the identity element is a $G(M)$-ambit, taking the  dense orbit to be $\{\tp(g/M)|\ g\in G(M)\}$. From now on, we will, throughout this paper, assume that every formula $\phi(x)\in L_\M$, with parameters in $\M$, is contained in $G(x)$, namely, the subset $\phi(\M)$ defined by $\phi$ is contained in $G$. Suppose that $\phi\in L_M$   and $g\in G(M)$, then the left translate $g\phi(x)$ is defined to be $\phi(g^{-1}x)$. It is easy to check that $(g\phi)(M)=gX$ with $X=\phi(M)$.
\begin{Def}
Let the notation be as above.
\begin{itemize}
\item A definable subset  $X\subseteq G$ is generic if finitely many left translates of $X$ cover $G$. Namely, there are $g_1,...,g_n\in G$ such that $G=\cup_{_{1\leq i\leq n}} g_iX$.
\item A definable subset  $X\subseteq G$ is weakly generic if there is a non-generic definable subset $Y\sq G$ such that $X\cup Y$ is generic.
\item A definable subset  $X\subseteq G$ is $f$-generic if for some/any model $M$ over which $X$ is defined and any $g\in G$, $gX$ does not divide over $M$. Namely, for any $M$-indiscernible sequence $(g_i:i< \omega)$, with $g=g_0$, $\{g_iX: i<\omega\}$ is consistent.
\item A formula $\phi(x)$ is generic if the definable set $\phi(\M)$ is generic. Similarly for weakly generic and $f$-generic formulas.
\item A type $p\in S_G(M)$ is generic if every formula  $\phi(x)\in p$ is generic. Similarly for weakly generic and $f$-generic types.
\item A type $p\in S_G(M)$ is almost periodic if $p$ is an almost periodic point of the $G(M)$-flow $S_G(M)$.
\item A global type $p\in S_G(\M)$ is strongly $f$-generic over a small model $M$ if every left $G(\M)$-translate of $p$ does not fork over $M$. A global type $p\in S_G(\M)$ is strongly $f$-generic if it is strongly $f$-generic over some small model.
\end{itemize}
\end{Def}

\begin{Fact}\cite{Newelski}
Let $AP\sq S_G(M)$ be the space of almost periodic types, and $WG\sq S_G(M)$ the space of weakly generic types. Then $WG=\cl(AP)$.
\end{Fact}

Recall that a type definable over $A$ subgroup $H\leq G$ has bounded index if $|G/H|<2^{|T|+|A|}$. For groups definable in $NIP$ structures, the smallest type-definable subgroup $G^{00}$ exists (see \cite{$NIP$}). Namely, the intersection of all type-definable subgroups of bounded index still has bounded index. We call $G^{00}$ the type-definable connected component of $G$. Another model theoretic invariant is $G^0$, called the definable-connected component of $G$,  which is the intersection of all definable subgroups of $G$ of finite index. Clearly, $G^{00}\leq G^0$.

Recall also that a  Keisler measure over $M$ on $X$,  with $X$ a definable subset of $M^n$, is a finitely additive probability measure on the Boolean algebra of $M$-definable subsets of $X$. When we take the monster model, i.e. $M=\M$, we call it a global Keisler measure. A definable group $G$ is said to be definably amenable if it admits a global (left) $G$-invariant Keisler measure.

The following facts will be used later.
\begin{Fact}\cite{CS}\label{amenable}
Let $G$ be a group definable in a saturated  $NIP$ structure $\M$.
 Then
\begin{enumerate}
  \item [(i)] $G$ is definably amenable if and only if it admits a global type $p\in S_G(\M)$ with
 bounded $G$-orbit.
  \item  [(ii)]  $G$ is definably amenable if and only if it admits a strongly $f$-generic type.
\end{enumerate}
\end{Fact}

Moreover,
\begin{Fact}\cite{CS}\label{f-generic-type-has-bounded-orbit}
Let $G$ be a  definably amenable group definable in a saturated  $NIP$ structure $\M$. Then
\begin{enumerate}
\item [(i)] $G$ admits a strongly $f$-generic type $p\in S_G(\M)$.
\item [(ii)] Weakly generic definable subsets, formulas, and types coincide with $f$-generic definable subsets, formulas, and types, respectively
\item [(iii)] $p\in S_G(\M)$ is $f$-generic if and only if it has bounded $G$-orbit.
\item [(iv)]  $p\in S_G(\M)$ is $f$-generic if and only if it is $G^{00}$-invariant.
\item [(v)]  The type-definable subgroup $H$ fixing a global $f$-generic type is exactly $G^{00}$.
\item [(vi)]  $G/G^{00}$ is isomorphic to the Ellis  group of $S_G(M)$ for any   $M\prec \M$.
\end{enumerate}
\end{Fact}

Note that in the context of Fact 1.6, if $p$ is global $f$-generic type which is either definable over, or finitely satisfiable in, a small model, then $p$ will be strongly $f$-generic (because of parts (iii) or (iv)).

\begin{Rmk}
Fact \ref{f-generic-type-has-bounded-orbit} (vi) gives a positive answer to the so-called Ellis group conjecture. The equality between $G/G^{00}$ and the Ellis group, under possibly additional assumptions, was suggested first  by Newelski \cite{Newelski} in Question 5.5 and Problem 5.6. A precise formulation under an assumption of $NIP$ and definable amenability was made by Pillay \cite{Pillay-TD}.  This  makes it reasonable to consider definably amenable  groups  as the ``stable-like" groups in $NIP$ environment.
\end{Rmk}

\subsection
{Groups definable in $(\Q,+,\times, 0,1)$}

We will be referring a lot to the comprehensive survey in \cite{Luc Belair} for the basic model theory of the p-adic field $\Q$. A key
point is Macintyre's theorem \cite{Macintyre} that $\Th(\Q,+,\times, 0,1)$ has quantifier elimination in the language where we add predicates $P_n(x)$ for the $n$-th powers for each $n\in \N^+$. The valuation is quantifier-free definable in this expanded language, in particular is definable
in the language of rings (see Section 3.2 of \cite{Luc Belair}).  Let us give some notation we will be using. $M$ denotes the structure $(\Q,+,\times, 0,1)$, $\Q^*=\Q\backslash \{0\}$ is the multiplicative group and $\mathbb Z$ is the ordered additive group of integers, the value group
of $\Q$. The group homomorphism $\nu: \Q^*\longrightarrow \mathbb Z$ is the valuation map.
The valuation map $\nu$ induces an absolute valuation $|\ \ |$ on $\Q$: for each $x\in \Q$, $|x|=p^{-\nu(x)}$ if $x\neq 0$ and $|x|=0$ otherwise. The absolute valuation makes the $p$-adic field $\Q$  into a locally compact topological field, with basis given by the sets $\nu(x-a)\geq n$ for $a\in \Q$ and $n\in \mathbb Z$. The ring of $p$-adic integers (valuation ring)
\[\Z=\{x\in \Q|\ x=0 \vee \ \nu(x)\geq 0\}
\] is of course
 a definable subset of $\Q$. We can add new predicates $\V(x)$ for the  valuation ring and $\V^{\times}(x)$ for the elements  with valuation $0$.  So $\V(M)=\V(\Q)=\Z$ and $\V^{\times}(M)=\V^{\times}(\Q)=\{x\in \Q| \nu(x)=0\}$.

 It is convenient to  refer to Section 3 of \cite{Pillay-On fields definable in Qp} for various notions such as definable $p$-adic analytic manifold and definable $p$-adic analytic group. (One can also consult \cite{O-P}.) In the background are the notions of dimension of definable (with parameters) sets in the structure $M = (\Q,+,\times, 0,1)$ as well as theorems on $p$-adic cell decomposition and definable functions, which are due to van den Dries and Scowcroft \cite{vdD-Scowcroft}. But we mention here a few details: For $X$ a definable set in $M$, say $X\subseteq \Q^{n}$, $dim(X)$ is the greatest $k\leq n$ such that some projection on $\Q^{k}$ contains an open set. An $n$-dimensional definable $p$-adic analytic manifold is a topological space with a covering by finitely many open sets each homeomorphic to an open definable (in $M$) subset of $\Q^{n}$ such that the transition maps  are definable and analytic.  See \cite{Serre} for analytic functions over complete fields. A definable $p$-adic analytic group is a definable $p$-adic analytic manifold equipped with a group structure which is definable and analytic when read in the appropriate charts. Such a definable $p$-adic analytic group is a definable group in $M$ and Lemma 3.8 of \cite{Pillay-On fields definable in Qp} says that conversely any group $G$ definable in $M$ can be definably equipped with the structure of a definable $p$-adic analytic group.

 In the current paper we will really be concerned with groups $G$ definable in the structure $M = (\Q,+,\times, 0,1)$.  However, the various notions introduced in the previous section depend on an ambient saturated model. So we let $\M$ denote a very saturated elementary extension $(\K, +, \times, 0, 1)$ of $M$. We also let $N=(K,+,\times, 0)\prec \M$ be an elementary extension of $M$.  So $K$ and $\K$ are $p$-adically closed fields.  The dimension of definable sets is defined as before, but can also be defined in terms of the underlying algebraic closure relation, which coincides with field-theoretic algebraic closure: assuming $X$ is definable over a finite set $A$, then $dim(X)$ can be described as the maximum of $dim({\bar a}/A)$ as ${\bar a}$ ranges over points of $X$.  We have the notion of a ``definable $C^{k}$-manifold in $N$ (or over $K$)", a topological space with a finite covering by open sets each homeomorphic to an open definable subset of $K^{n}$ (fixed $n$) with transition maps definable and $C^{k}$. Either adapting the methods of \cite{Pillay-On fields definable in Qp} or by transfer from the case $K = \Q$, one sees that for any group $G$ definable in $N$ and for any $k<\omega$, $G$ can be definably equipped with the structure of a definable $C^{k}$-manifold in $N$ with respect to which the group structure is $C^{k}$.  In contrast with the analogous situation for real closed fields, it seems not to be known whether we can also do it for $k = \infty$. On the other hand if $G$ is defined with parameters from $\Q$, then $G(\Q)$ is as in the previous paragraph, and so the definable analytic manifold and group structure on $G(\Q)$ DOES give rise to a definable $C^{\infty}$-structure on $G$ with respect to which the group operation is $C^{\infty}$.

Recall from \cite{O-P} that a definable manifold $X$ over $K$ is definably compact, if for any definable continuous function $f : \V(K) \backslash\{0\} \longrightarrow X$, the limit $\displaystyle\lim_{x \to 0}f(x)$  exists in X. Definable compactness agrees with compactness if $K$ is $\Q$. A definable subset of $K^n$ with
the induced topology  is definably compact if and only if it is closed and bounded.

Let $G$ be a group definable in $N$ with parameters from $\Q$.  As remarked above $G(\Q)$ has (definably in $M$) the structure of a $p$-adic analytic (Lie) group, and as such has an open compact subgroup, which must be definable in $M$. It follows that $G$ has a $\Q$-definable subgroup $C$ with $dim(C) = dim(G)$ and $C(\Q)$ is compact (and $C$ definably compact).

Let $C$ be a group definable in $N$ over $\Q$ which is (with its definable manifold topology) definably compact. Then $C(\Q)$ is compact.  We have the standard part map $\st: C\longrightarrow C(\Q)$. The  kernel  of $\st$ is precisely the group of infinitesimals of $C$ and coincides with the intersection of all $\Q$-definable subgroups of $C$ of finite index.
This is elaborated on in the following Fact.

\begin{Fact}\label{fsg-G00=G0}\label{fsg=compact}
Let $C$ be a group definable in $\M$ over $\Q$.
\begin{enumerate}
  \item [(i)] If $C$  definably compact, then $C^{00}=C^0$ coincides with $\Ker(\st)$, and $\st$ induces a homeomorphism
between $C/C^{0}$ (with its logic topology) and the p-adic Lie group $C(\Q)$.
  \item  [(ii)] $C$ is definably compact iff $C$ has $fsg$.
\end{enumerate}
\end{Fact}
\begin{proof} For (i) see Corollary 2.4 of \cite{O-P}.
\newline
The right implies left direction of (ii) is Corollary 2.3 (iv) of \cite{O-P}. The left to right direction (easy direction) does not appear explicitly in \cite{O-P}, but appears in somewhat greater generality in Proposition 3.1 of \cite{Johnson-fsg}

\end{proof}

\begin{Fact}\cite{P-Y2}\label{one-dim-definably-amenable}
If $G$ is a one-dimensional group definable over $\Q$, then $G$ is abelian-by-finite, hence definably amenable.
\end{Fact}


\subsection{Definable Groups and Quotient Groups}

We will try to give some coherent notation regarding definable groups and algebraic groups. Our underlying first order theory is $Th(\Q)$ in the ring language (or in the Macintyre language, a definitional expansion). As in Section 1.4, $M$ denotes the standard model, $\M$ a ``monster model", and $N\prec\M$ a not necessarily saturated elementary extension of $M$. When we speak of a group definable {\em in} $M$, $N$ or $\M$, we mean the obvious thing. When we speak of a group $G$ definable {\em over}
$\Q$ (for example) we typically mean a group definable in $\M$ defined with parameters from $\Q$. In this case $G(\Q)$ denotes the group definable in $\Q$ by the same formulas defining $G$ in $\M$.
On the other hand if for example $G$ is a group definable {\em in } $\Q$ then we can consider the groups definable in $N$ or $\M$ by the same formulas, which we may refer to as $G(K)$, $G(\K)$, or $G(N)$, $G(\M)$.
In general, when we speak of a definable object (set, or group) we mean a definable object in the monster model $\M$ of $Th(\Q)$. As mentioned in Section 1.4 we are mainly interested in groups definable in the standard model $M$, but we will need information about their interpretations in $\M$.
Let us again emphasize that we distinguish between definable and interpretable (definable in $(\K)^{eq}$) and our groups will be definable.

The usual mathematical notation for algebraic groups is to consider an algebraic group $G$ over a field $K$ as something like a functor which takes a field $L$ containing $K$ to the group $G(L)$. Alternatively, one can  identify  $G$ with $G(L)$ for a given {\em algebraically closed field} containing $K$, where the variety structure as well as the group structure are given by data (polynomial equations, transition maps, morphisms) over $K$. (See \cite{Pillay-ACF} for more details on this and the later discussion of algebraic varieties and groups over algebraically closed fields.)
With this notation the algebraic groups we will consider will be over $\Q$ and we will consider their groups $G(\Q)$ and $G(\K)$ of $\Q$-points, and $\K$-points. Of course the latter will be also definable groups in the structures $M$, $\M$ respectively, in the earlier sense, but essentially just quantifier-free definable in the ring language.

So in this paper we will slightly modify notation, by defining a  {\em $p$-adic algebraic group} to be the group $G(\Q)$ of $\Q$-points of an algebraic group $G$ over $\Q$.  When $G$ is a linear algebraic group over $\Q$, then $G(\Q)$ will simply be a subgroup of some $GL_{n}(\Q)$ defined by a finite system of polynomial equations over $\Q$. So a $p$-adic algebraic group is a special case of a group definable in $M$. As above we denote by $G(\K)$ the group of $\K$-points of $G$, which is consistent with our earlier notation regarding definable groups.

We now just want to observe that the quotient of a $p$-adic algebraic group by a $p$-adic algebraic subgroup can be seen as a definable (rather than just interpretable) set in $M$: So we are given algebraic groups $G, H$ over $\Q$ such that $H(\Q)$ is a subgroup of $G(\Q)$. We assume $G$ to be connected as an algebraic group. Let $L$ be an algebraically closed field containing $\Q$ such as $\Q^{alg}$. Then $H(L)$ is a subgroup of $G(L)$. But then by elimination of imaginaries in $ACF$ together with a theorem of Weil, $G(L)/H(L)$ is, in the structure $(L,+,\times)$, definably isomorphic over $\Q$ to an algebraic variety $X(L)$ over $\Q$, which is a homogeneous space
for $G(L)$ all defined over $\Q$. Namely we have a surjective morphism $f:G(L) \to X(L)$ defined over $\Q$, constant on cosets of $H(L)$ in $G(L)$ and inducing a bijection between $G(L)/H(L)$ and $X(L)$. Restricting $f$ to $G(\Q)$ gives a definable (in $M$) bijection between $G(\Q)/H(\Q)$ and a subset $Y$ of $X(\Q)$ which is definable in $M$ and can be seen, by dimension reasons, to be an open subset of the $p$-adic manifold $X(\Q)$, so a $p$-adic manifold itself.  In the case where $H(\Q)$ is normal in $G(\Q)$, then $Y$ becomes an open subgroup of the $p$-adic algebraic group $X(\Q)$.

\newpage

\section{Main Results}

\subsection{$fsg$ groups and $dfg$ groups over $\Q$}


\begin{Fact}\label{G-finite-to-one-H}
Let $G$ be a group definable in $\M$ over $\Q$ and suppose $G$ to be definably amenable.  Then there is an algebraic group $H$ over  $\Q$ and a finite-to-one $\Q$-definable group homomorphism from  $G^{00}$ to $H(\K)$.
\end{Fact}
\noindent
{\em Explanation.}  This follows from Theorem 2.19 of \cite{MOS}, making use of the proof of Corollary 2.22 of \cite{MOS}.



\begin{Lemma}\label{f-to-1}
Let $G$ be a definably amenable group over $\Q$ such that $G^{00}=G^{0}$. Then there is a  $\Q$-definable  subgroup $A\leq G(\Q)$ of finite index and a finite normal subgroup $A_0\sq A$ such that  $A/A_0$ is $\Q$-definably isomorphic to an open  subgroup of a $p$-adic algebraic group.
\end{Lemma}
\begin{proof}

By Fact \ref{G-finite-to-one-H}, there is an algebraic group $H$ over $\Q$ and a definable over $\Q$ finite-to-one group homomorphism from   $G^{00}$ to  $H(\K)$. Since $G^{00}=G^{0}$, by compactness, there is a $\Q$-definable finite index subgroup $A\leq G$   and a $\Q$-definable finite-to-one group homomorphism $f$ from  $A(\Q)$ to $H(\Q)$. Let $H_0(\Q)$ be the Zariski closure of $\im(f)$ in $H(\Q)$, then $H_0(\Q)$ is a $p$-adic algebraic group.   
By Remark 2.13 of \cite{Hrushovski-Pillay}, $\dim(\im(f))=\dim(H_0(\Q))$. By the topological definition of dimension $im(f)$ has interior in $H_{0}(\Q)$, hence $im(f)$, being a subgroup, is open in $H_{0}(\Q)$.
\end{proof}

\begin{Fact}\cite{P-Y}\label{dfg-implies-G0=G00}
Assuming $NIP$. If a definable group $G$ has $dfg$, then $G^{00}=G^{0}$.
\end{Fact}


\begin{Cor}\label{fsg-dfg-f-to-1}
Let $G$ be  a group definable over $\Q$ which is either $fsg$ or $dfg$.  Then there is a finite index $\Q$-definable subgroup $A\leq G(\Q)$ and a finite nornmal subgroup $A_0\sq A$ such that  $A/A_0$ is $\Q$-definably isomorphic to an  open subgroup of a $p$-adic algebraic group.
\end{Cor}
\begin{proof}
By Fact \ref{fsg-G00=G0}, Fact \ref{dfg-implies-G0=G00}, and Lemma \ref{f-to-1}.
\end{proof}

\subsection{One-dimensional $dfg$ groups over $\Q$}

Recall that the dp-rank of a partial type $\pi(x)$ over $A$ is $\geq \kappa$ if there are $a\models \pi(x)$ and $\kappa$  sequences $(I_i,i<\kappa)$ mutually indiscernible over $A$ and none of which is indiscernible over $Aa$. Dp-minimal theories are theories in which all
nonalgebraic $1$-types have dp-rank
$1$. Every dp-minimal theory has $NIP$.

\begin{Fact}\cite{Dolich-Goodrick-Lippel}\label{Q-p-is-dp-minimal}
$\Th(\Q)$ is dp-minimal.
\end{Fact}

\begin{Fact}\cite{Simon-dp-minimal}\label{dp-rank-1}
Assume that $T$ has $NIP$. Let $N\models T$ be a small model. If $p$ is an $N$-invariant global type of dp-rank $1$, then $p$ is either finitely satisfiable in $N$ or  definable over $N$.
\end{Fact}
\begin{Fact}\cite{Simon-dp-minimal}\label{dp-rank=acl-dim}
Assume that $T$ is dp-minimal, and $\acl$ satisfies exchange. Then  dp-rank coincides with the $\acl$-dimension.
\end{Fact}

\begin{Fact}\cite{$NIP$II}\label{nonforking=invariant}
If $T$ has $NIP$ then a global complete type does not fork over a small submodel $M$ iff it is $M$-invariant.
\end{Fact}

\begin{Lemma}
Let $G$ be a $\Q$-definable group of dimension $1$. Then $G$ is either $dfg$ or $fsg$.
\end{Lemma}
\begin{proof}
By Fact \ref{one-dim-definably-amenable}, $G$ is definably amenable. By Fact \ref{f-generic-type-has-bounded-orbit}, $G$ admits a strongly $f$-generic type. Let $p\in S_G(\K)$ be a global strongly $f$-generic type. There is a small submodel $M$ such that  every left translate of $p$ does not fork over $M$, and hence $M$-invariant by Fact \ref{nonforking=invariant}. By Fact \ref{Q-p-is-dp-minimal} and Fact \ref{dp-rank-1},  $p$ is either definable over $M$ or finitely satisfiable in $M$. So $G$ is either $dfg$ or $fsg$.
\end{proof}

Recall the notion of generically stable type from \cite{$NIP$II}: a global type $p$ is generically stable if it is both definable and finitely satisfiable
in some small model.

\begin{Fact}\cite{Simon-dis-nodis}\label{no-generical-stable-types}
Let $p\in S(\K)$ be a non-algebraic global type. Then $p$ is not generically stable.
\end{Fact}


\begin{Cor}\label{dfg-implies-non-compact}

If $G$ is an infinite $\Q$-definable group  which has  $dfg$ then  $G$ is NOT definably compact.
\end{Cor}
\begin{proof} Suppose for a contradiction (using Fact \ref{fsg=compact}) that $G$ were both $dfg$ and $fsg$.  Let $p$ be a  global definable (strongly) $f$-generic type. By \cite{HPS2} $p$ is generic, namely every formula in $p$ is generic. By Proposition 4.2 of \cite{$NIP$}, $p$ (and all its translates) are finitely satisfiable in a small model. So $p$ is generically stable, which contradicts Fact \ref{no-generical-stable-types}. This proves the Corollary.

\end{proof}

The following is from \cite{PPY}. Part (a) is Lemma 2.1 there, Part (b)(i) is contained in Proposition 2.3 there, and part (b)(ii) in Proposition 2.4.  Let us emphasize that the data in  items (ii) and (iii) of part (a) determine complete types over $M$ as is proved in \cite{PPY}.
\begin{Fact}\label{complete 1 types over Qp}
(a) The complete $1$-types over $M$ (or $\Q$) are precisely the following:
\begin{enumerate}
  \item [(i)] The realized types $\tp(a/M)$ for each $a\in \Q$.
  \item  [(ii)] For each $a\in \Q$ and coset $C$ of $(\K^{*})^0$, the type $p_{a,C}$ saying that $x$ is infinitesimally
close to a (i.e. $\nu(x-a)>n$ for each $n\in \N$), and $(x-a)\in C$.
  \item  [(iii)] For each coset $C$ as above the type $p_{\infty, C}$ saying that $x\in C$ and $\nu(x)<n$ for all $n\in \mathbb Z$.
  \end{enumerate}
  (b)
  \begin{enumerate}
    \item  [(i)]
    A global type of $\Ga(\K)$ is (strongly) $f$-generic iff it is  an heir of some $p_{\infty, C}$.
  \item  [(ii)] A global type of $\Gm(\K)$ is (strongly) $f$-generic iff it is  an heir of some $p_{\infty, C}$ or $p_{_{0,C}}$.
\end{enumerate}
\end{Fact}


Part (ii) of the next remark is precisely Remark 2.2 (ii) from \cite{PPY}

\begin{Rmk}\label{Rmk-complete 1 types over Qp}
\begin{enumerate}
  \item [(i)]  Let $q(x)\in S_1(\K)$ be a $\emptyset$-definable $1$-type. By Fact \ref{complete 1 types over Qp}, we see that  if $q$ is NOT infinitesimally
close to any point of $\Q$ over $\Q$, then $q$ is a global heir of some $p_{\infty, C}$.
  \item  [(ii)]
  Conversely, any global complete $1$-type extending the partial type $\nu(x)<\Gamma$ is an heir of some $p_{\infty, C}$, and hence definable over $\Q$.
\end{enumerate}

\end{Rmk}



\begin{Lemma}\label{limit-point-is-definable-f-generic-type}
Suppose that $G$ is a one-dimensional group definable over $\Q$.  Let
$f: \Z\backslash\{0\}\longrightarrow
G(\Q)$ be a definable continuous function such that
$\displaystyle\lim_{x \to 0}f(x)\notin G(\Q)$, and
$\alpha\in \bar \K\succ \K$ such that $\nu(\alpha)>\Gamma$. Then $\tp(f(\alpha)/\K)$ is a global
definable $f$-generic type of $G$.
\end{Lemma}
\begin{proof}
First let us check that the notation in the statement of the Lemma makes sense.  First the continuity of
$f$ is meant with respect to the $p$-adic topology on $\Z$ and the definable $p$-adic analytic group
topology on $G(\Q)$ from Section 1.4.  For any elementary extension $N = (K,+,\times,0)$ of $M =
(\Q,+,\times,0)$, $f(K)$ is a continuous definable over $\Q$ function from $\V\setminus\{0\}$ to $G(K)$, which we may also write as just $f$.  So for $\alpha\in \bar \K\succ \K$, nonzero such that $\nu(\alpha)>\Gamma$, $f(\alpha) \in G(\bar K)$.

Note that $\tp(\alpha/\K)$ is $\emptyset$-definable. Since $\displaystyle\lim_{x \to 0}f(x)\notin G(\Q)$, $f(\alpha)$ is not infinitesimally close (over $\Q$) to any point of $G(\Q)$. Moreover, since $G(\Q)$ has a definable manifold structure, for any $g\in G(\Q)$, $gf(\alpha)$ is not infinitesimally close to any point of $G(\Q)$ over $\Q$.

By Lemma \ref{f-to-1}, let $H(\Q)$ be a one-dimensional algebraic group, $ U(\Q)\leq G(\Q)$  a
$\Q$-definable subgroup of finite index,  and  $\pi:U(\Q)\ra H(\Q)$ a definable finite-to-one homomorphism. Let $V(\Q)=\pi(U(\Q))$. Since $U(\Q)$ is generic, we may assume that $f(\alpha)\in U(\bar \K)$. Let $\beta=\pi(f(\alpha))$. Then, for any $g\in V(\Q)$, $g\beta$ is not infinitesimally close to any point of $V(\Q)$ over $\Q$.

\begin{Claim}
$\beta$ is not infinitesimally close to any point of $H(\Q)$.
\end{Claim}
\begin{proof}
As any open subgroup is clopen, we see that $V(\Q)$ is a clopen subgroup of $H(\Q)$.
Now suppose for a contradiction that $\beta$ is infinitesimally close to $h\in H(\Q)$ over $\Q$\blue{, then} $h\in \cl(V(\Q))=V(\Q)$. A contradiction.
\end{proof}
Now $H(\Q)$ is a $p$-adic algebraic group (as defined in Section 1.5). So $H(\Q)$ is  definably isomorphic to $(\Q,+)$, or $(\Q^*,\times)$, or a definably compact group $C(\Q)$ (which is either an anisotropic group or an elliptic curve). But any point in $C(\bar \K)$ is infinitesimally close to its standard part in $C(\Q)$ over $\Q$. So $H(\Q)$ is isomorphic to $(\Q,+)$, or $(\Q^*,\times)$.

As  $\beta$ is not infinitesimally close to any point of $H(\Q)$, and $\tp(\beta/\K)$ is definable over $\emptyset$,
we conclude by Remark \ref{Rmk-complete 1 types over Qp} and \ref{complete 1 types over Qp}  that  $tp(\beta/\K)$ is the heir of  some $p_{\infty,C}$ in the case $H = (\Q,+)$, and the heir of some $p_{\infty,C}$ or some $p_{0,\C}$ in the case that $H = (\Q^{*},\times)$.  Either way $tp(\beta/\K)$ is a global definable (strong) $f$-generic of $H$.
So the $V(\K)$-orbit of $\tp(\beta/\K)$ is bounded. Since $\pi: U(\K)\longrightarrow V(\K)$ is a finite-to-one homomorphism,  the $U(\K)$-orbit of $\tp(f(\alpha)/\K)$ is bounded. This implies that $\tp(f(\alpha)/\K)$ is a global $f$-generic type of $G$.
\end{proof}


\begin{Prop}\label{dfg-subgroup-has-finite-index}
Let $G$ be a one-dimensional $\Q$-definable group. If $H\sq G$ is an open $dfg$ subgroup of $G$ definable over $\Q$,  then $H$ has finite index in $G$.
\end{Prop}
\begin{proof}
By Corollary \ref{dfg-implies-non-compact}, there is a  definable continuous function $f: \Z\backslash\{0\}\rightarrow H(\Q)$  such that $\displaystyle\lim_{x \to 0}f(x)\notin H(\Q)$. Moreover, we claim that
\begin{Claim}
There is  a definable continuous function $f: \Z\backslash\{0\}\rightarrow H(\Q)$  such that $\displaystyle\lim_{x \to 0}f(x)\notin G(\Q)$.
\end{Claim}
\begin{proof}
Suppose NOT. Then $H(\Q)=\cl(H(\Q))$ is definably compact.  This contradicts  Corollary \ref{dfg-implies-non-compact}.
\end{proof}
Let $f$ be a definable continuous function $f: \Z\backslash\{0\}\rightarrow H(\Q)$  such that $\displaystyle\lim_{x \to 0}f(x)\notin G(\Q)$, and $\alpha\in \bar \K\succ \K$ such that $v(\alpha)>\Gamma$. By Lemma \ref{limit-point-is-definable-f-generic-type}, $\tp(f(\alpha)/\K)$ is a global definable $f$-generic type of both $G$ and $H$. Now $\tp(f(\alpha)/\K)\in S_H(\K)$ is $G^{00}$-invariant. So $G^{00}\sq H$. By compactness, $H$ has finite index.
\end{proof}

\begin{Cor}
Let $G$ be a one-dimensional  $dfg$  group definable over $\Q$. Then $G^{00}=\cap_{n\in \N^+}nG$.
\end{Cor}
\begin{proof}
By Fact \ref{one-dim-definably-amenable}, we may assume that $G$ is commutative. Consider the map $f_n: G\ra nG$ given by $x\mapsto nx$. Let $\tp(\alpha/\K)\in S_G(\K)$ be a definable $f$-generic type. Then there is a small submodel $M\prec \K$ such that every left $nG$-translate of $\tp(n\alpha/\K)$ is definable over $M$. So $nG$ is a $dfg$ group. By Proposition \ref{dfg-subgroup-has-finite-index}, $nG$ has finite index for each $n\in \N^+$. On the other hand, if $H\leq G$ is a definable subgroup of index $n$ with $n\in \N^+$\blue{, then} $nG\leq H$. So $G^{0}=\cap_{n\in \N^+}nG$. By Fact \ref{dfg-implies-G0=G00}, we have $G^{00}=\cap_{n\in \N^+}nG$ as required.
\end{proof}

Recall that a definable group is definably connected if it has no proper finite index definable subgroups. We will call a definable group $G$ is definably totally disconnected if  every finite index definable subgroup of $G$ is NOT definably connected. Namely, a definably totally disconnected group has no minimal  finite index definable subgroups.

Note that $\Ga(\K)$  is definably connected since $\Ga(\K)^0=\Ga(\K)$, while $\Gm(\K)$ is definably totally disconnected since $\Gm(\K)^0=\bigcap_{n\in \N^+}P_n(\K^*)$ (see \cite{PPY}).

\begin{Prop}
Let $G$ be a one-dimensional  $dfg$  group definable over $\Q$.
\begin{enumerate}
  \item If $G(\Q)$ is definably connected, then there is a finite normal subgroup $A_0$ such that $G(\Q)/A_0$ is definably isomorphic to $(\Q,+)$;
  \item If $G(\Q)$ is definably totally disconnected, then there is a definable subgroup $A\leq G$ of finite index, and a finite normal subgroup $A_0\leq A(\Q)$ such that $A(\Q)/A_0$ is definably isomorphic to $(P_n(\Q^*),\times)$
\end{enumerate}
\end{Prop}
\begin{proof}
By Lemma \ref{f-to-1}, there are a one-dimensional algebraic group $H$ and a definable finite-to-one group homomorphism $f$ from $A(\Q)$ to $H(\Q)$, where $A\leq G$ is a finite index subgroup definable over $\Q$. Clearly, $A$ has $dfg$. Let $\tp(\alpha/\K)$ be a global definable $f$-generic type of $A$. Then $\tp(f(\alpha)/\K)$ is also a  global definable $f$-generic type of  $\im(f)$. As $\im(f)$ has $dfg$,  $\im(f)\leq H$ has finite index by Proposition  \ref{dfg-subgroup-has-finite-index}. So $H(\Q)$ is $dfg$, and hence is either $(\Q,+)$ or $(\Q^*,\times)$.

If $G(\Q)$ is definably connected.   Then $\im(f)$ is definably connected. So $H(\Q)$ has to be $(\Q,+)$, and hence is $\im(f)$.

If $G(\Q)$ is totally disconnected, then $\im(f)$ is also totally disconnected. So $H(\Q)$ has to be $(\Q^*,\times)$, and hence  $\im(f)$   contains $(P_n(\Q^*),\times)$ for some $n$. Replace $A(\Q)$ by $f^{-1}(P_n(\Q^*))$ if necessary.
\end{proof}

\subsection{Algebraic groups with open $dfg$ subgroups}
The following will be used later:
\begin{Fact}\cite{O-P}\label{open subset}
Let $X\sq \K^n$ be open, definable and defined over $\Q$. Let $X=Y_1\cup Y_2$, where the $Y_i$ are definable in $\K$. Then one of the $Y_i$ contains an open $\Q$-definable subset.
\end{Fact}

\begin{Lemma}\label{dfg group has no compact parts}
Let $H$ be an algebraic group over $\Q$, and $G$ an open $\Q$-definable subgroup of $H(\K)$. Suppose $A$ is an algebraic subgroup of $H$, also defined over $\Q$. Let $G/A(\K)$ be the set of cosets $g/A(\K)$ for $g\in G$. Then $G/A(\K)$ infinite implies $G/A(\K)$ is not definably compact (equivalently $G(\Q)/A(\Q)$ is not compact.
\end{Lemma}
\begin{proof}
We start with some explanations, partly suggested by the referee. First the quotient space $H(\Q)/A(\Q)$
(a homogeneous space for $H(\Q)$) can be viewed as a  definable (rather than interpretable) set in $M$, and moreover a definable $p$-adic manifold. This is by the last paragraph of Section 1.5. So $G(\Q)/A(\Q)$, as defined in the statement of the lemma
will be a definable submanifold of $H(\Q)/A(\Q)$, and a homogeneous space for $G(\Q)$, and we call it $C(\Q)$. So $C = C(\K)$ is a definable (over $\Q$) manifold in the sense of Section 1.4, and also a homogeneous space for $G = G(\K)$.
Note also that the lemma we are proving is a natural generalization of \ref{dfg-implies-non-compact}
to homogeneous spaces (transitive flows).

Let $\pi: G\to C$ be the natural $\Q$-definable projection map.

So we have the action of $G = G(\M)$ on $C = C(\M)$ and hence also an action on the type space $S_{C}(\M)$ of global types concentrating on $C$.  We will call a definable subset $X$ of $C$ {\em generic} if finitely many $G$-translates of $X$ cover $C$, and likewise a type $p\in S_{C}(\M)$ generic if all formulas in $p$ are generic.



We first  show that for any $\Q$-definable open subset $O(\M)$ of $G(\M)$, $\pi(O(\M))$ contains a $\Q$-definable open  subset of $C(\M)$. Take $g\in O(\M)$ such that $\dim(g/\Q)=\dim(G)$\blue{, then} we have $g\in \pi^{-1}(\pi(g))$ and \[
\dim(\pi^{-1}(\pi(g)))=\dim(gA)=\dim(A).\]
So $\dim(g/\Q,\pi(g))\leq \dim(A)$, and thus
\[
\dim(\pi(g)/\Q)=\dim(g,\pi(g)/\Q)-\dim(g/\Q,\pi(g))\geq \dim(G)-\dim(A).
\]
As $\pi(g)\in\pi(O(\M))$, we see that
\[
\dim(O(\M))\geq \dim(\pi(g)/\Q)\geq  \dim(G)-\dim(A)=\dim(H/A)\geq \dim(C).
\] So $\pi(O(\M))$ contains a $\Q$-definable open subset of $C(\M)$ as required.

\vspace{5mm}
\noindent
Now suppose for a contradiction that $C(\Q)$ is definably compact. Then $C(\Q)$ is compact  as a topological space (as mentioned earlier, definable compactness and compactness agree for definable $p$-adic manifolds).

\begin{Claim}
$S_C(\M)$ has a generic type which is moreover finitely satisfiable in $\Q$ (i.e. in $C(\Q)$).
\end{Claim}
\begin{proof}
Let $p(x)$ be any type in $S_{C}(M)$ which contains only open definable sets, and let $p'(x)\in S_{C}(\M)$ be a coheir of $p$. We will show that $p'$ is a generic type of $S_{C}(\M)$ as defined above.
So let $X$ be a definable set in $p'$. By definability of types over $\Q$, $X\cap C(\Q) = Y$ is a definable set in the structure $M$. Note that $Y$ has interior in the sense of the space $C(\Q)$, because otherwise $Y\notin p$, so $X\wedge\neg(Y(\M))\in p'$, but the latter definable set does not meet $C(\Q)$,
contradicting $p'$ being finitely satisfiable in $M$. We have shown that $Y$ has interior in $C(\Q)$. It follows that $Y(\M)$ contains a $\Q$-definable open set.

Now consider  $Y(\M) = (X\cap Y(\M)) \cup ((\neg X)\cap Y(\M))$. Note that $(\neg(X)\cap Y(\M))\cap C(\Q)) = \emptyset$ so in particular does not contain any $\Q$-definable open set. It follows from
\ref{open subset} that $X\cap Y(\M)$ contains a $\Q$-definable open set, so in particular $X$ contains a $\Q$-definable open set $Z$. As $C(\Q)$ is compact, finitely many $G(\Q)$ translates of $Z(\Q)$ cover $C(\Q)$. So  finitely many $G$-translates of $Z$ cover $C$, so finitely many $G$-translates of $X$ cover $C$. So $X$ is generic, and as $X$ was an arbitrary definable set in $p'$ we have shown that $p'$ is generic. This completes the proof of the Claim.    \end{proof}

Let us now complete the proof of the lemma. Note that we have a $G(\M)$-flow map from $(G(\M), S_{G}(\M))$ to $(G(\M), S_{C}(\M))$ induced by $\pi$.  By our assumptions, let $r\in S_{G}(\M)$ be a definable $f$-generic type. By Lemma 1.15 of \cite{P-Y} (and its proof) the orbit $(G(\M)\cdot r$ is closed, hence a minimal subflow of $S_{G}(\M)$. Then $\pi(G(\M)\cdot r)$ is a minimal subflow of $S_{C}(\M)$. But $S_{C}(\M)$ has generic types, by the Claim, hence as is well-known (see \cite{Newelski}) has a unique minimal subflow.  By the Claim, let $p\in S_{C}(\M)$ be generic and finitely satisfiable in $C(\Q)$. But then $p$ is of the form $\pi(q)$ for some $q\in G(\M)\cdot r$. As $q$ is definable, so is $\pi(q) = p$. Then $p$ is both definable and finitely satisfiable, which contradicts distality.     \end{proof}

\begin{Rmk}
Note that if $G$ and $H$ are definable groups, and  $G$ has $dfg$, and $f: G\lra H$ is a definable onto homomorphism, then $H$ has $dfg$.
\end{Rmk}


At this point we will need some of the basic theory of linear algebraic groups. The notation in the literature (such as \cite{J.S. Milne}) for (linear) algebraic groups over a field $k$ is consistent with our notation from Section 1.5 where $k = \Q$. See Chapter 14 of \cite{J.S. Milne} for the notions of unipotent elements in a linear algebraic group $G$ and unipotent algebraic subgroups $G$.
An algebraic group $G$ over $\Q$ is {\em trigonalizable} over $\Q$ if it is isomorphic over $\Q$ to a group of upper triangular matrices in some $GL_{n}$. An algebraic torus over $\Q$ is split over $\Q$ it is isomorphic over $\Q$ to a product of copies of the multiplicative group.
A connected linear solvable algebraic group $G$ over $\Q$ is {\em split} over $\Q$ if $G$ admits a subnormal series of algebraic subgroups over $\Q$ whose quotients are isomorphic over $\Q$ to the additive or multiplicative group.
A connected semisimple algebraic group $G$ over $\Q$ is said to be {\em isotropic} if contains a $\Q$-split algebraic torus defined over $\Q$. Let $G$ be a reductive algebraic group over $\Q$. A parabolic subgroup is a connected subgroup $P$ of $G$ such that the quotient variety $G/P$ is complete.

We summarise the various facts from the literature that we need (specialized to the case $k = \Q$).
\begin{Fact}\label{structure theorem}
\begin{enumerate}
  \item [(i)] (see Theorem 16.33 of \cite{J.S. Milne}.) Let $G$ be a connected linear solvable algebraic group over $\Q$. Then $G$ has a maximal unipotent algebraic subgroup, $G_{u}$ which is normal and over $\Q$ and $G/G_{u}$ is an algebraic torus (over $\Q$).
  \item   [(ii)] (see Proposition 16.52 of \cite{J.S. Milne}.)
  Any connected solvable linear algebraic group over $\Q$ which is split over $\Q$ is trigonalizable over $\Q$.
  \item  [(iii)]   (see Corollaire 4.17 of \cite{Borel-Tits}) Let $G$ be a (connected) semisimple algebraic group over $\Q$. Then $G$ is $\Q$-isotropic if and only if $G$ has a proper parabolic subgroup $P$ defined over $\Q$.
  \item [(iv)] (Theorem 3.1 of \cite{V. Platonov-A. Rapinchuk}.)  Suppose $G$ is a connected semisimple algebraic group over $\Q$ which is $\Q$-anisotropic. Then $G(\Q)$ is compact.
\end{enumerate}
\end{Fact}


\begin{Prop}\label{algebraic groups with dfg subgroup}
Let $H$ be a connected (not necessarily linear) algebraic group over $\Q$. Suppose that $H(\M)$ contains an open $\Q$-definable $dfg$ subgroup $G$, then $H$ is linear, solvable, and the algebraic torus $H/H_{u}$ (from part (i) of the Fact above) splits over $\Q$.
\end{Prop}
\begin{proof}
The algebraic group $H$ admits a short exact sequence of connected algebraic groups over $\Q$
\[
1\lra L\lra H\lra A\lra 1
\]
with $L$ a linear algebraic, and $A$ an abelian variety.
Now $G(\Q)/L(\Q)$ definably over $\Q$ embeds in $A(\Q)$ which is compact, so  $G(\Q)/L(\Q)$ is compact so definably compact.
If $A$ is nontrivial, then $dim(L) < dim(H)$, so as $G(\Q)$ is open in $H(\Q)$, $G(\Q)/L(Q)$ is infinite.
By Lemma \ref{dfg group has no compact parts} this implies that  $G(\Q)/L(\Q)$ is not definably compact, so not compact, a contradiction. So $A$ is trivial and hence $H$ is linear.

Let $R = \rad(H)$ be the solvable radical of $H$, the maximal connected normal solvable algebraic subgroup of $H$, which is automatically defined over $\Q$.  Let $S$ be the quotient $H/R$ a connected semisimple algebraic group over $\Q$. We want to show that $S$ is trivial, so $H$ will be solvable, as desired.
Suppose for a contradiction that $S$ is nontrivial.

As in  the statement of Lemma 2.19,  we define $G/R(\K)$ to be the set of cosets $g/R(\K)$ in $H(\K)/R(\K)$ for $g\in G$, which (as $R$ is normal in $H$ and $G$ an open $dfg$ subgroup of $H(\K)$) is an open $\Q$-definable $dfg$ subgroup of $S(\K)$. As $G/R(\K)$ is infinite, by Lemma \ref{dfg group has no compact parts} $G/R(\K)$ is not definably compact, so $G(\Q)/R(\Q)$ is not compact. But the latter is a closed subgroup of $S(\Q)$, whereby $S(\Q)$ is not compact. By Fact \ref{structure theorem} (iv), $S$ is $\Q$-isotropic. So by Fact \ref{structure theorem} (iii), $S$ has a proper parabolic subgroup $P$ defined over $\Q$.  As $S/P$ is a complete variety defined over $\Q$, $(S/P)(\Q)$ is compact, so $(S/P)(\K)$ is definably compact.  Let $X$ be the image of $G$ (or of $G/R(\K)$) in $(S/P)(\K)$. Then $X$ is also definably compact, again contradicting Lemma \ref{dfg group has no compact parts}.  This contradiction proves that $S$ was trivial, and so $H = R$ is solvable.

We finally have to prove that the toric part $T = H/H_{u}$ of $H$ splits over $\Q$.  $T$ is an algebraic torus defined over $\Q$, and by p. 53 of \cite{V. Platonov-A. Rapinchuk} for example, $T$ is an almost direct product of tori $T_{1}$, $T_{2}$ defined over $\Q$ where $T_{1}$ is $\Q$-split and $T_{2}$ has no nontrivial split subtorus over $\Q$. We have to show that $T_{2}$ is trivial.  Otherwise $T_{2}$ is nontrivial, and by Theorem 3.1 of \cite{V. Platonov-A. Rapinchuk} $T_{2}(\Q)$ is compact so $T_{2}(\K)$ is definably compact. Now  $H$ has a proper normal algebraic subgroup $H_{1}$ defined over $\Q$ (namely the preimage of $T_{1}$ under the homomorphism $H \to H/H_{u}$) such that the quotient $H/H_{1} = T_{2}$.
But by \ref{dfg group has no compact parts}, as usual the image of $G$ in $(H/H_{1}(\K)$ is infinite and not definably compact, a contradiction. This completes the proof of the Proposition.
\end{proof}



\begin{Cor}\label{algebraic groups with dfg subgroup is trigonalizable over Qp}
Let $H$ be a connected algebraic group over $\Q$. If $H(\K)$ contains an open $\Q$-definable $dfg$ subgroup $G$ say, then $H$ is trigonalizable over $\Q$.
\end{Cor}
\begin{proof}  By the previous proposition $H$ is linear, solvable, and $H/H_{u}$ is a $\Q$-split torus.
As we are in characteristic $0$ the unipotent part $H_{u}$ of $H$ is split (i.e. admits a subnormal series of algebraic subgroups over $\Q$ whose quotients are isomorphic over $\Q$ to the additive group).
Hence $H$ itself is a $\Q$-split solvable group over $\Q$. By Fact\label{structure theorem} (ii), $H$ is trigonalizable over $\Q$.
\end{proof}

\subsection{$dfg$ Groups and  $f$-generic types}

Let $A$, $B$, $C$ be  definable groups, defined over $\Q$, such that $A\leq B$ and $\pi: B\longrightarrow C$ is $\Q$-definable, such that moreover
\[
1\longrightarrow A \xlongrightarrow{i}  B \xlongrightarrow{\pi}   C \longrightarrow 1
\]
is a short exact sequence, where $i: A\longrightarrow B$ is the inclusion map. Since $\Th(\Q)$ has definable Skolem functions (see \cite{Dries-skolem-functions}), let  $f: C\longrightarrow B$ be a  $\Q$-definable section of $\pi$. Every element $b$ of $B$ can be written uniquely as
$af(\pi(b))$ for some $a\in A$.

We will be freely using Facts 1.5 and 1.6 in the following.

\begin{Lemma}\label{tp(pi(b)) and tp(a) are f-generic}
Fix notation as above. Suppose that $B$ is definably amenable, then both $A$ and $C$ are definably amenable. Moreover if $b\in B(\bar \M)$ is such that $\tp(b/\M)$ is an $f$-generic type, then both $\tp(\pi(b)/\M)$ and $\tp(a/\M)$ are $f$-generic types of $C$ and $A$ respectively
\end{Lemma}
\begin{proof}
Since $\tp(b/\M)$ is $B^{00}$-invariant, we see that $\tp(\pi(b)/\M)$ is $\pi(B^{00})$-invariant, so $C$ is definably amenable and $\tp(\pi(b)/\M)$ is an $f$-generic type of $C$.

Let $\eta: B\rightarrow A$ be the function given by $\eta(x)=x\cdot f(\pi(x))^{-1}$, then $a=\eta(b)$. For each $a_0\in A$, we have
\[
a_0\cdot a=a_0\cdot b\cdot f(\pi(b))^{-1}=(a_0\cdot b)\cdot f(\pi(a_0\cdot b))^{-1}=\eta(a_0\cdot b)\]
Since $\tp(b/\M)$ is $f$-generic, we see that
\[
A\cdot \tp(b/\M)=\{\tp(a_0\cdot b/\M)|\ a_0\in A\}
\]
is bounded. So  the $A$-orbit  of $\tp(a/\M)$
\[
\{\tp(a_0\cdot a/\M)|\ a_0\in H\}=\{\tp (\eta(a_0\cdot b)/\M)|\ a_0\in A\}
\]
is bounded. We conclude that $A$ is definably amenable and  $\tp(a/\M)$ is an $f$-generic type of $C$.


\end{proof}

\begin{Lemma}\label{both A and C are dfg}
Let the notation be as above. Suppose that $B$ has $dfg$. Then both $A$ and $C$ have $dfg$.
\end{Lemma}
\begin{proof}
Let $\tp(b/\M)$ be a global $f$-generic type  definable over some small model $M_0$. Then, by Lemma \ref{tp(pi(b)) and tp(a) are f-generic}, we see that both $\tp(\pi(b)/\M)$ and $\tp(a/\M)$ are $f$-generic types of $C$ and $A$ respectively.

Assume that $M_0$ contains the parameters of $A,B,C$ and $\pi$, then $a,\pi(b)\in \dcl(M_0,b)$, we see that $\tp(\pi(b)/\M)$ and $\tp(a/\M)$ are  definable over $M_0$, and hence $dfg$ types of $C$ and $A$, respectively. This concludes the proof that  $A$ and $C$ have $dfg$.
\end{proof}

\begin{Rmk}
Let $A,B$ be algebraic groups over $\Q$ such that $A$ is a subgroup of $B$, and $B(\K)$ is $dfg$. Since the coset space $(B/A)(\K)$ is definable, a similar argument to the proof of Lemma \ref{both A and C are dfg} shows that $A(\K)$ is also  $dfg$.
\end{Rmk}

\begin{Prop}\label{open dfg subgroup of alg groups}
Let $H$ be an algebraic group over $\Q$ and $G$ an open $dfg$ subgroup of $H(\K)$ definable over $\Q$. Then $G$  has finite index in $H(\K)$
\end{Prop}
\begin{proof}
The proof is by induction on $dim(H)$. The statement holds for $\dim(H)=1$ by Proposition \ref{dfg-subgroup-has-finite-index}.

Now assume that $\dim(H)>1$.

By  Corollary \ref{algebraic groups with dfg subgroup is trigonalizable over Qp}, $H$ is trigonalizable over $\Q$, and hence there is a short exact sequence of algebraic groups over $\Q$
\[
1\longrightarrow A \xlongrightarrow{i}  H \xlongrightarrow{\pi}   C \longrightarrow 1,
\]
where $A$ an $1$-dimensional algebraic group central in $H$, and $C$ an algebraic group of dimension $\dim(H)-1$. Then we have an induced short exact sequence:
\[
1\longrightarrow A(\K)\cap G \xlongrightarrow{i}  G \xlongrightarrow{\pi}   \pi(G) \longrightarrow 1
\]
 By Lemma \ref{both A and C are dfg}, we see that both $A(\K)\cap G$ and $\pi(G)$ are open $dfg$ subgroups of $A(\K)$ and $C(\K)$ respectively. By induction hypothesis,  both $A(\K)\cap G$ and $\pi(G)$ have finite index in $A(\K)$ and $C(\K)$ respectively.  This concludes that $G$ has finite index in $H(\K)$.
\end{proof}

\begin{Cor}\label{aa}
Suppose that $G_1\leq G_2$ are $dfg$ groups definable over $\Q$. If $\dim(G_1)=\dim(G_2)$, then $G_1$ has finite index in $G_{2}$.
\end{Cor}
\begin{proof}
By Corollary \ref{fsg-dfg-f-to-1}, there are a finite index subgroup $A$ of $G_{2}$  and a finite normal subgroup $A_0$ of $A$, with both $A$ and $A_{0}$ definable over $\Q$, and such that  $A/A_0$ is $\Q$-definably isomorphic to an open subgroup of $H(\K)$ for some algebraic group $H$ defined over $\Q$.  Let $f: A\lra H(\K)$ be the finite-to-one homomorphism with $\ker (f)=A_0$
as in Fact  \ref{G-finite-to-one-H}. By Proposition \ref{open dfg subgroup of alg groups}, $\im(f)$ has finite index in $H(\Q)$. As $G_1\cap A$ has finite index in $G_1$, we see that $G_1\cap A$ has $dfg$. So the image $f(G_1\cap A)$ is an open $dfg$ subgroup of $H(\K)$, and hence has finite index in $H(\K)$ by Proposition \ref{open dfg subgroup of alg groups}. It follows easily that $G_{1}$ has finite index in $G_{2}$.
\end{proof}

A generalization of Corollary \ref{aa} has been proved in \cite{AJ}, Corollary 4.4.

\vspace{5mm}
\noindent
We now prove the following
two structure theorems  for $dfg$ groups using Proposition \ref{open dfg subgroup of alg groups}:

\begin{Thm}\label{normal sequence of open dfg subgroups}
Let $G$ be an open $dfg$ $\Q$-definable subgroup of $H(\K)$ where $H$ is an algebraic group defined over $\Q$. Then  there is a normal sequence
\[
\{\id_G\}=G_0\vartriangleleft ... \vartriangleleft G_i \vartriangleleft G_{i+1} \vartriangleleft... \vartriangleleft G_n=G
\]
such that each $G_{i}$ is definable over $\Q$ and each $G_{i+1}/G_i$ is definably over $\Q$ isomorphic to a one-dimensional $dfg$ group.
\end{Thm}
\begin{proof}
Induction on $\dim(G)$. Clearly, the statement is trivial if $\dim(G)=1$. Now assume that $\dim(G)>1$.
As we showed in Proposition \ref{open dfg subgroup of alg groups},
there is a short exact sequence
\[
1\longrightarrow A \xlongrightarrow{i}  H \xlongrightarrow{\pi}   C \longrightarrow 1,
\]
with $A$ a one-dimensional algebraic group central in $H$, and $C$ an algebraic group of dimension $\dim(H)-1$, where everything is over $\Q.$ Both $A(\K)\cap G$ and $\pi(G)$ are open $dfg$ subgroups of $A(\K)$ and $C(\K)$ respectively. By induction hypothesis, there is  a normal sequence
\[
\{\id_C\}=D_0\vartriangleleft ... \vartriangleleft D_i \vartriangleleft D_{i+1} \vartriangleleft... \vartriangleleft D_n=\pi(G)
\]
such that each $D_{i+1}/D_i$ is definably isomorphic to a one-dimensional $dfg$ group over $\Q$. It is easy to see that the normal sequence
\[
\{\id_G\}\leq G_0=G\cap A(\K)\vartriangleleft ... \vartriangleleft  G_{i+1}=G\cap\pi^{-1}(D_i) \vartriangleleft... \vartriangleleft G_{n+1}=G
\]
meets our requirements.
\end{proof}


\begin{Thm}
Let $G$ be a group definable over $\Q$ with $dfg$. Then  there is a normal sequence
\[
G_0\vartriangleleft ... \vartriangleleft G_i \vartriangleleft G_{i+1} \vartriangleleft... \vartriangleleft G_n
\]
of $\Q$-definable groups, such that $G_0$ is finite, $G_n$ is a finite index subgroup of $G$, and  $G_{i+1}/G_i$ is isomorphic to a one-dimensional $dfg$ group over $\Q$.
\end{Thm}
\begin{proof}
By Corollary \ref{fsg-dfg-f-to-1}, there is a finite index subgroup $A_{1}\leq G$ and a finite normal subgroup $A_0\sq A$, all $\Q$-definable,  such that  $A_1/A_0$ is $\Q$-definably isomorphic to an  open subgroup of $H(\K)$ for $H$ some (connected)  algebraic group defined over $\Q$. Let $f: A_1\lra H(\K)$ be the finite-to-one homomorphism with $\ker (f)=A_0$ as in Fact \ref{G-finite-to-one-H}. By Proposition \ref{open dfg subgroup of alg groups}, $\im(f)$ has finite index in $H(\K)$. Moreover $H$ is trigonalizable over $\Q$ by Proposition \ref{algebraic groups with dfg subgroup}.

Clearly, the statement is trivial if $\dim(G)=1$. Now assume that $\dim(G)>1$. As we showed in Proposition \ref{open dfg subgroup of alg groups},
there is a short exact sequence
\[
1\longrightarrow A \xlongrightarrow{i}  H \xlongrightarrow{\pi}   C \longrightarrow 1,
\]
with $A$ a one-dimensional algebraic group central in $H$, and $C$ an algebraic group of dimension $\dim(H)-1$, all defined over $\Q$. Both $A(\K)\cap \im(f)$ and $\pi(\im(f))$ are open $dfg$ subgroups of $A(\K)$ and $C(\K)$ respectively. By Theorem \ref{normal sequence of open dfg subgroups}, there is a normal sequence
\[
\{\id_C\}=D_0\vartriangleleft ... \vartriangleleft D_i \vartriangleleft D_{i+1} \vartriangleleft... \vartriangleleft D_n=\pi(\im(f))
\]
such that each $D_{i+1}/D_i$ is definably isomorphic to a one-dimensional $dfg$ group over $\Q$. Let $E_i=\pi^{-1}(D_i)\cap \im (f).$ It is easy to see that $D_i=\pi(E_i)$, thus \blue{$E_i/(A(\K)\cap \im(f))\cong D_i$}. So
\blue{\[
E_{i+1}/E_{i}\cong\big(E_{i+1}/(A(\K)\cap \im(f))\big)/\big(E_i/(A(\K)\cap \im(f))\big)\cong D_{i+1}/D_i
\]}
is definably isomorphic to a one-dimensional $dfg$ group over $\Q$. We conclude that
\[
E_0\vartriangleleft  ... \vartriangleleft E_{i} \vartriangleleft E_{i+1} \vartriangleleft... \vartriangleleft E_n=\im(f)
\]
is a normal sequence such that $ E_0=A(\K)\cap \im(f)$ and each $E_{i+1}/E_i$ is definably isomorphic to a one-dimensional $dfg$ group over $\Q$.
Let $G_0=A_0$ and $G_{i+1}=f^{-1}(E_i)$, we see that the normal sequence
\[
G_0\vartriangleleft ... \vartriangleleft G_i \vartriangleleft G_{i+1} \vartriangleleft... \vartriangleleft G_{n+1}=A_{1}
\]
satisfies our requirement.
\end{proof}

Finally we answer the Conjecture \ref{Main Conjecture 2} in the $p$-adic case, making use of:

\begin{Fact}\cite{Y-Trigonalizable}\label{trigonalizable algebraic  group almost periodic}
Let $H$ be a trigonalizable algebraic  group over $\Q$. Then every global complete $f$-generic type of $H(\K)$ is almost periodic.
\end{Fact}

By Proposition \ref{open dfg subgroup of alg groups} and Fact \ref{trigonalizable algebraic  group almost periodic}, we easily conclude that
\begin{Thm}
Suppose that $G$ is a $dfg$ group over $\Q$. Then every global complete $f$-generic type of $G$ is almost periodic.
\end{Thm}

\section*{Acknowledgements}
The authors would like to thank the referee for carefully reading our paper and offering detailed comments, which have been very helpful for us as we revise the paper.

The first author was supported by NSF grants
DMS-1665035, DMS-1760212 and DMS-2054271. The second author
was supported by The National Social Science Fund of China (Grant No. 20CZX050).

\end{document}